\documentclass[12pt,reqno,a4paper]{amsart}

\usepackage{amsmath,amsfonts,amsthm,amssymb,amsxtra,enumerate,mathtools,mathabx}
\usepackage[normalem]{ulem}
\usepackage[utf8]{inputenc}
\usepackage[T1]{fontenc}
\usepackage{lmodern}
\usepackage{bbm}
\usepackage{mathrsfs}
\usepackage{enumerate}
\usepackage{comment}
\usepackage{placeins}
\usepackage{pgfplots}
\pgfplotsset{compat=newest}
\usepgfplotslibrary{fillbetween}
\usepackage{csquotes}
\usepackage{tikz}
\usetikzlibrary{arrows.meta, positioning}

\usepackage{float}

\usepackage{blindtext}

\makeatletter

\newtheorem{theorem}{Theorem}
\newtheorem{proposition}[theorem]{Proposition}
\newtheorem{lemma}[theorem]{Lemma}

\theoremstyle{definition}

\theoremstyle{remark}

\newcommand{\C}{\mathbb{C}}

\newcommand{\dd}{\, \mathrm{d}}

\renewcommand{\epsilon}{\varepsilon}

\renewcommand{\phi}{\varphi}
\newcommand{\R}{\mathbb{R}}

\newcommand{\Hs}{\mathcal{H}}

\DeclareMathOperator{\re}{Re}

\DeclareMathOperator{\Tr}{Tr}
\DeclareMathOperator{\tr}{Tr}

\let\oldtocsection=\tocsection
\let\oldtocsubsection=\tocsubsection
\let\oldtocsubsubsection=\tocsubsubsection
\renewcommand{\tocsection}[2]{\hspace{0em}\oldtocsection{#1}{#2}}
\renewcommand{\tocsubsection}[2]{\hspace{0.75cm}\oldtocsubsection{#1}{#2}}
\renewcommand{\tocsubsubsection}[2]{\hspace{2em}\oldtocsubsubsection{#1}{#2}}

\date{14th September 2026}
\numberwithin{equation}{section}
\usepackage[hidelinks]{hyperref}

\begin{document}
	
	\begin{abstract}
		We prove that the best constant in the one-dimensional Lieb--Thirring inequality with exponent one is $4/(3\sqrt{3}\pi)$, confirming the Lieb--Thirring conjecture in this case. Moreover, we extend the inequality to operator-valued potentials and obtain the bound $L_{1,d}\leq(2/\sqrt3)L_{1,d}^{\mathrm{cl}}$ in every dimension.
	\end{abstract}
	
	\title[A sharp one-dimensional Lieb--Thirring inequality]{The sharp one-dimensional Lieb--Thirring inequality for the sum of eigenvalues}
	
	\author[Larry Read]{Larry Read}
	\address[Larry Read]{Laboratoire de mathématiques d’Orsay, CNRS, Université Paris-Saclay, 91405 Orsay, France}
	\email{larry.read@universite-paris-saclay.fr}
	\author[Marvin R. Schulz]{Marvin R. Schulz}
	\address[Marvin R. Schulz]{Department of Mathematical Sciences, University of Copenhagen, Universitetsparken 5, 2100 Copenhagen, Denmark}
	\email{masc@math.ku.dk}
	
	\maketitle
	
	For a one-dimensional Schr\"odinger operator $-d^2/dx^2-V$, the lowest negative eigenvalue $\lambda_1$ satisfies the sharp bound \cite{keller1961lower,lieb1976inequalities}
	\begin{align*}
		|\lambda_1|\leq\frac4{3\sqrt3\pi}\int_{\R}V(x)_+^{3/2}\dd x\,,
	\end{align*}
	where $V_+=\max\{V,0\}$. Lieb and Thirring \cite{lieb1976inequalities} conjectured that the same estimate holds with $|\lambda_1|$ replaced by the sum of the absolute values of all negative eigenvalues. In this paper we prove this conjecture.
	
	\begin{theorem}\label{thm:main}
		Let $V\in L^{3/2}(\R)$ be real-valued. Then the negative eigenvalues $\{\lambda_k\}_{k\geq1}$ of $-d^2/dx^2-V$ satisfy
		\begin{align*}
			\sum_{k\geq1}|\lambda_k|\leq\frac4{3\sqrt3\pi}\int_{\R}V(x)_+^{3/2}\dd x\,.
		\end{align*}
		The constant is sharp.
	\end{theorem}
	
	Equality holds for $V(x)=\frac34\operatorname{sech}^2x$, which produces a single negative eigenvalue, $\lambda_1=-1/4$.
	
	This question belongs to the general problem of determining the optimal constants in the celebrated Lieb--Thirring inequalities
	\begin{align*}
		\sum_{k\geq1}|\lambda_k|^\gamma\leq L_{\gamma,d}\int_{\R^d}V(x)_+^{\gamma+d/2}\dd x\,,
	\end{align*}
	where $\{\lambda_k\}_{k\geq1}$ are the negative eigenvalues of $-\Delta-V$ in $L^2(\R^d)$, counted with multiplicity, and $L_{\gamma,d}$ denotes the optimal constant. For real-valued $V\in L^{\gamma+d/2}(\R^d)$, Lieb and Thirring \cite{lieb1976inequalities} proved these bounds for $\gamma>1/2$ if $d=1$ and $\gamma>0$ if $d\geq2$. The endpoints $\gamma=1/2$, $d=1$, and $\gamma=0$, $d\geq3$, were established by Weidl \cite{weidl1996lieb} and Cwikel, Lieb, and Rozenblum \cite{cwikel1977weak,lieb1976bounds,rozenbljum1972distribution,rozenblum1976distribution}, respectively.
	
	The lowest eigenvalue and the semiclassical limit give two competing lower bounds on $L_{\gamma,d}$. The first is the sharp constant $L_{\gamma,d}^{(1)}$ for the lowest eigenvalue alone; the opening inequality gives $L_{1,1}^{(1)}=4/(3\sqrt3\pi)$. The second is the semiclassical constant
	\begin{align*}
		L_{\gamma,d}^{\mathrm{cl}}=\frac{\Gamma(\gamma+1)}{(4\pi)^{d/2}\Gamma(\gamma+1+d/2)}\,,
	\end{align*}
	which follows from Weyl asymptotics in the strong-coupling limit. The original conjecture of Lieb and Thirring proposes that the optimal constant is determined by one of these two regimes: a single bound state, or the semiclassical limit in which the number of bound states tends to infinity. In terms of the constants, it asserts that
	\begin{align*}
		L_{\gamma,d}=\max\big\{L_{\gamma,d}^{(1)},L_{\gamma,d}^{\mathrm{cl}}\big\}\,.
	\end{align*}
	
	Until now, the conjecture had been proved only for $\gamma\geq3/2$ in every dimension and for $\gamma=1/2$ in one dimension. The one-dimensional case $\gamma=3/2$ predates Lieb and Thirring's work and follows from the KdV trace identities \cite{buslaev1960traces,zakharov1971korteweg,gardner1974korteweg}; Aizenman--Lieb monotonicity \cite{aizenman1978semi} gives all $\gamma>3/2$. Laptev and Weidl \cite{laptev2000sharp} established the result for $d>1$, while Hundertmark, Lieb, and Thomas \cite{hundertmark1998sharp} proved $L_{1/2,1}=L_{1/2,1}^{(1)}=1/2$. In one dimension, the single-eigenvalue constant is the larger candidate for $1/2\leq\gamma<3/2$, and the two constants coincide at $\gamma=3/2$. Numerical investigations support the one-dimensional conjecture \cite{levitt2014best} and suggest periodic optimizers in two dimensions \cite{frank2021periodic}.
	
	By contrast, the conjecture fails in several higher-dimensional ranges, including $d=3$, $1/2<\gamma<1$ \cite{helffer1990riesz,frank2021nonlinear}. We refer to the surveys \cite{frank2020lieb,schimmer2022state,frank2023orthonormal,nam2023direct} and the monograph \cite{frank2022schrodinger} for further background.
	
	The exponent $\gamma=1$ is of particular importance: the dual kinetic inequality is the key estimate in Lieb and Thirring's proof of the stability of matter \cite{lieb1975kinetic}. In one dimension, Eden and Foias \cite{eden1991simple} obtained $L_{1,1}\leq2/(3\sqrt3)$, and Dolbeault, Laptev, and Loss \cite{dolbeault2008improved} extended their argument to matrix-valued potentials. Building on Rumin's direct kinetic method \cite{rumin2011balanced}, Frank, Hundertmark, Jex, and Nam \cite{frank2021revisited} proved $L_{1,d}\leq1.456L_{1,d}^{\mathrm{cl}}$ in every dimension. Refining their method, Corso and Ried \cite{corso2025variational} obtained
	\begin{align*}
		L_{1,d}\leq1.44655L_{1,d}^{\mathrm{cl}}\,,\qquad d\geq1\,,
	\end{align*}
	which gives the best previously known upper bound for $L_{1,1}$. In every dimension, Nam \cite{nam2018semiclassical} obtained a kinetic bound with constant arbitrarily close to the semiclassical value and a gradient error term. Seiringer and Solovej \cite{seiringer2023simple} subsequently strengthened this estimate and gave a simpler proof.
	
	Benguria and Loss \cite{benguria2004ovals} established a remarkable connection between the two-eigenvalue bound and the ovals conjecture, a spectral isoperimetric problem. This conjecture has been studied in \cite{burchard2005isoperimetric,linde2006lower,denzler2015existence,bernstein2015projective,linde2025improved}. Our proof extends their elementary one-eigenvalue argument, based on cumulative mass, to orthonormal families of arbitrary finite size, retaining the sharp constant.
	
	The argument also extends to vector-valued orthonormal families. By duality \cite{dolbeault2008improved}, this gives the one-dimensional Lieb--Thirring inequality for operator-valued potentials, with the same sharp constant. The dimension lifting argument of Laptev and Weidl \cite{laptev2000sharp} then gives the following bound.
	
	\begin{theorem}\label{thm:higher-dimensional}
		For every $d\geq1$, we have
		\begin{align*}
			L_{1,d}\leq\frac2{\sqrt3}L_{1,d}^{\mathrm{cl}}\,.
		\end{align*}
	\end{theorem}
	
	The factor $2/\sqrt3\approx 1.1547$ improves the previously best bound of $1.44655$ by Corso and Ried \cite{corso2025variational}.
	We prove Theorem~\ref{thm:higher-dimensional} in
	Section~\ref{sec:vector-proof}.
	
	\section{Proof of Theorem~\ref{thm:main}}
	
	In this section, we work with real-valued functions. By the standard duality between the spectral and kinetic Lieb--Thirring inequalities \cite[Section~2]{schimmer2022state}, it suffices to prove the following bound for every integer $N\geq1$ and every family $u_1,\ldots,u_N\in H^1(\R)$ orthonormal in $L^2(\R)$:
	\begin{align}\label{eqn:kinetic}
		\sum_{j=1}^N\int_{\R}(u_j')^2\dd x \geq\frac{\pi^2}{4}\int_{\R}\left(\sum_{j=1}^N u_j^2\right)^3\dd x\,.
	\end{align}
	
	Our starting point is the sharp scalar inequality for real-valued functions due to Sz.-Nagy \cite{sznagy1941integral},
	\begin{align}\label{eqn:scalar1part}
		\int_{\R}(u')^2\dd x\geq\frac{\pi^2}{4}\int_{\R}u^6\dd x\,,\qquad u\in H^1(\R)\,,\qquad \|u\|_2=1\,.
	\end{align}
	Benguria and Loss \cite[Section~2]{benguria2004ovals} obtain this inequality by using cumulative mass as a new variable, reducing it to the Dirichlet Poincar\'e inequality on $(0,1)$. We rewrite their argument using an auxiliary function $q$, in a form that will allow us to extend it to orthonormal families.
	
	Let $q\in C^1([0,1])$ and $\kappa>0$ satisfy
	\begin{align*}
		q'(s)+q(s)^2\leq-\kappa\,,\qquad 0\leq s\leq1\,.
	\end{align*}
	For $u\in H^1(\R)$ with $\|u\|_2=1$, let
	\begin{align*}
		a(x)=\int_{-\infty}^x u(y)^2\dd y\,.
	\end{align*}
	Then $0\leq a\leq1$ and $a'=u^2$. By the product rule and our assumption on $q$, we have almost everywhere
	\begin{align*}
		(u')^2-\left(\frac14u^4q(a)\right)'=\left(u'-\frac12u^3q(a)\right)^2-\frac14u^6\left(q'(a)+q(a)^2\right)\geq \frac{\kappa}4 u^6\,.
	\end{align*}
	Since $u$ tends to zero at both ends of the line and $q$ is bounded on $[0,1]$, the derivative of $\frac14u^4 q(a)$ integrates to zero. Integrating the inequality therefore gives
	\begin{align*}
		\int_{\R}(u')^2\dd x\geq\frac\kappa4\int_{\R}u^6\dd x\,.
	\end{align*}
	
	To recover the sharp constant, we take $q=b_\alpha$, where
	\begin{align*}
		b_\alpha(s)=-\alpha\tan\bigl(\alpha(s-1/2)\bigr)\,,\qquad 0<\alpha<\pi\,.
	\end{align*}
	Then $b_\alpha$ is smooth on $[0,1]$ and satisfies $b_\alpha'+b_\alpha^2=-\alpha^2$, so letting $\alpha\uparrow\pi$ gives \eqref{eqn:scalar1part}. The limiting weight $\pi\cot(\pi s)$ is the logarithmic derivative of $\sin(\pi s)$, the first Dirichlet eigenfunction on $(0,1)$, which arises from the cumulative-mass change of variables in the argument of Benguria and Loss.

	We now extend the argument to an orthonormal family. Write $U=(u_1,\ldots,u_N)^{\mathsf T}$ and denote its density and kinetic energy density by $\rho=|U|^2$ and $\tau=|U'|^2$, respectively. Define the cumulative mass matrix
	\begin{align*}
		A(x)=\int_{-\infty}^x U(y)U(y)^{\mathsf T}\dd y\,.
	\end{align*}
	For every unit vector $c$, the quantity $c\cdot A(x)c$ is the $L^2$ mass of the normalised function $c\cdot U$ on $(-\infty,x)$. Consequently, $0\leq A\leq I$ and $A'=UU^{\mathsf T}$.
	
	For the orthonormal family we again introduce an auxiliary function $q\in C^1([0,1])$ and a constant $\kappa\in\R$, whose relation will be specified below. We seek a locally absolutely continuous function $F_q$ that replaces $\frac14u^4q(a)$ and satisfies
	\begin{align*}
		\tau-F_q'\geq\frac\kappa4\rho^3\quad\text{a.e. on }\R\,,\qquad \lim_{x\to\pm\infty}F_q(x)=0\,.
	\end{align*}
	These properties give the kinetic bound with constant $\kappa/4$ by integration.
	
	We will use
	\begin{align}\label{eqn:boundary-term}
		F_q=\frac14\sum_{i,j}p_ip_jq\left(\frac{\lambda_i+\lambda_j}{2}\right)\,,
	\end{align}
	where, at each $x$, the vectors $e_i$ form an orthonormal eigenbasis of $A(x)$ with eigenvalues $\lambda_i$, and $p_i=(e_i\cdot U(x))^2$. This recovers $\frac14u^4q(a)$ when $N=1$. For polynomial $q$, we will rewrite $F_q$ as a polynomial in $A$ and $U$, establishing its local absolute continuity without differentiating eigenvectors.
	
	\begin{lemma}\label{lem:polynomial}
		Suppose a function $q\in C^1([0,1])$ and a constant $\kappa>0$ satisfy
		\begin{align}\label{eqn:certificate}
			q(s)q(t)+\frac{q(s)-q(t)}{s-t}\leq-\kappa\,,\qquad 0\leq s,t\leq 1\,.
		\end{align}
		Then
		\begin{align}\label{eqn:polynomial-bound}
			\sum_{j=1}^N\int_{\R}(u_j')^2\dd x \geq\frac\kappa4\int_{\R}\left(\sum_{j=1}^N u_j^2\right)^3\dd x\,.
		\end{align}
	\end{lemma}
	
	\begin{proof}
		We write $q[s,t]=(q(s)-q(t))/(s-t)$, with $q[s,s]=q'(s)$.
		
		It suffices to prove the result for polynomials. Indeed, approximate $q'$ uniformly by real polynomials $r_n$ and set $q_n(s)=q(0)+\int_0^s r_n(t)\dd t$. Then $q_n\to q$ in $C^1([0,1])$, and
		\begin{align*}
			\sup_{s,t\in[0,1]}|q_n[s,t]-q[s,t]|\leq\|q_n'-q'\|_{L^\infty(0,1)}\longrightarrow 0\,.
		\end{align*}
		Thus $q_n(s)q_n(t)+q_n[s,t]\leq-\kappa+\varepsilon_n$ uniformly on $[0,1]^2$, for some $\varepsilon_n\geq0$ tending to zero. Applying the polynomial case with $\kappa-\varepsilon_n$ and passing to the limit gives \eqref{eqn:polynomial-bound}, since both integrals are finite for $U\in H^1(\R;\R^N)$.
		
		Introduce the scalar moments $M_r=U\cdot A^rU$ and $N_r=U'\cdot A^rU$ for integer $r\geq 0$. The ordinary product rule and $A'=UU^{\mathsf T}$ give
		\begin{align}\label{eqn:moment-rule}
			M_r'=2N_r+\sum_{\ell=0}^{r-1}M_\ell M_{r-1-\ell}\,.
		\end{align}
		Here $(A^r)'=\sum_{\ell=0}^{r-1}A^\ell A'A^{r-1-\ell}$, and substituting $A'=UU^{\mathsf T}$ gives $U\cdot A^\ell A'A^{r-1-\ell}U=M_\ell M_{r-1-\ell}$. The two outer derivatives both equal $N_r$ because $A^r$ is symmetric. Expand the symmetric polynomial $q((s+t)/2)$ and introduce $F_q$ by
		\begin{align*}
			q\left(\frac{s+t}{2}\right)=\sum_{r,m}c_{rm}s^rt^m\,,\qquad c_{rm}=c_{mr}\,,\qquad F_q=\frac14\sum_{r,m}c_{rm}M_rM_m\,.
		\end{align*}
		We verify below that this expression agrees with \eqref{eqn:boundary-term} at each point. Since $A$ and $U$ are locally absolutely continuous, so are the moments $M_r$ and $F_q$; all derivative identities below hold almost everywhere.
		
		All sums are finite. Differentiating and using the symmetry of the coefficients yields
		\begin{align}\label{eqn:moment-derivative}
			F_q'=\sum_{r,m}c_{rm}N_rM_m+\frac12\sum_{r,m}c_{rm}\sum_{\ell=0}^{r-1}M_\ell M_{r-1-\ell}M_m\,.
		\end{align}
		
		Indeed, interchanging $r$ and $m$ shows that the two terms from differentiating $M_rM_m$ have equal sums, so $F_q'=\frac12\sum_{r,m}c_{rm}M_r'M_m$. Substituting \eqref{eqn:moment-rule} gives the coefficients $1$ and $1/2$ above.
		
		At the point where we evaluate this identity, choose any orthonormal eigenbasis $e_i$ of $A$, with $Ae_i=\lambda_i e_i$. Set
		\begin{align*}
			v_i=e_i\cdot U\,,\qquad w_i=e_i\cdot U'\,,\qquad p_i=v_i^2\,,\qquad s_{ij}=\frac{\lambda_i+\lambda_j}{2}\,,\qquad B_i=\sum_jp_jq(s_{ij})\,.
		\end{align*}
		Then $M_r=\sum_i p_i\lambda_i^r$, $N_r=\sum_i v_iw_i\lambda_i^r$, $\sum_i p_i=\rho$ and $\sum_iw_i^2=\tau$. Substituting the formula for $M_r$ into the polynomial definition of $F_q$ recovers \eqref{eqn:boundary-term}.
		
		Subtracting the expansions of $q((s+t)/2)$ and $q((z+t)/2)$ and using
		\begin{align*}
			s^r-z^r=(s-z)\sum_{\ell=0}^{r-1}s^\ell z^{r-1-\ell}\,,\qquad r\geq1\,,
		\end{align*}
		gives
		\begin{align*}
			\frac{q((s+t)/2)-q((z+t)/2)}{s-z}=\sum_{r,m}c_{rm}\sum_{\ell=0}^{r-1}s^\ell z^{r-1-\ell}t^m\,.
		\end{align*}
		Substituting this into \eqref{eqn:moment-derivative} and using $s_{ij}-s_{kj}=(\lambda_i-\lambda_k)/2$ gives
		\begin{align}\label{eqn:chain-rule}
			F_q'=\sum_i v_iw_iB_i+\frac14\sum_{i,j,k}p_ip_jp_kq[s_{ij},s_{kj}]\,.
		\end{align}
		In the first term, summing $c_{rm}\lambda_i^r\lambda_j^m$ recovers $q(s_{ij})$. For the second term, the midpoint gives
		\begin{align*}
			\frac{q(s_{ij})-q(s_{kj})}{\lambda_i-\lambda_k}=\frac12q[s_{ij},s_{kj}]\,.
		\end{align*}
		This changes the coefficient $1/2$ in \eqref{eqn:moment-derivative} to $1/4$, matching the coefficient produced by the square below. The quotient is interpreted continuously when $\lambda_i=\lambda_k$. Subtracting \eqref{eqn:chain-rule} from $\tau$ and completing the squares gives
		\begin{align}\label{eqn:tauFrho}
			\tau-F_q'&=\sum_i\left(w_i-\frac12v_iB_i\right)^2-\frac14\sum_{i,j,k}p_ip_jp_k\left(q(s_{ij})q(s_{kj})+q[s_{ij},s_{kj}]\right)\\
			&\geq\frac\kappa4\rho^3\,.\notag
		\end{align}
		Indeed, the cross term is $-\sum_i v_iw_iB_i$. Expanding $\sum_i p_iB_i^2$ and interchanging $i,j$ gives the displayed product of values of $q$, paired with the divided difference from differentiating $A$. Their common coefficient $1/4$ is what allows \eqref{eqn:certificate} to apply. Every $s_{ij}$ belongs to $[0,1]$, and the non-negative weights $p_ip_jp_k$ sum to $\rho^3$.
		
		We now integrate \eqref{eqn:tauFrho} over $(-\Lambda,\Lambda)$ and let $\Lambda\rightarrow \infty$. Since $F_q$ is locally absolutely continuous, we get 
		\begin{align*}
			\int_{-\Lambda}^\Lambda\tau\dd x-F_q(\Lambda)+F_q(-\Lambda)\geq\frac\kappa4\int_{-\Lambda}^\Lambda\rho^3\dd x\,.
		\end{align*}
		To show that $F_q(\Lambda)$ and $F_q(-\Lambda)$ go to zero as $\Lambda\rightarrow \infty$, we use
		\begin{align*}
			F_q=\frac14\sum_{i,j}p_ip_jq(s_{ij})\,,\qquad |F_q|\leq\frac14\|q\|_{L^\infty(0,1)}\rho^2\,.
		\end{align*}
		Here we used $p_i\geq0$ and $\sum_{i,j}p_ip_j=(\sum_i p_i)^2=\rho^2$. Since $U\in H^1(\R;\R^N)$, it is bounded, and $\rho'=2U\cdot U'\in L^1$ gives limits of $\rho$ at both ends, necessarily zero because $\rho\in L^1$. Thus $F_q(\pm\infty)=0$, while $\int\rho^3\leq\|\rho\|_\infty^2\int\rho<\infty$ and $\tau\in L^1$. Letting $\Lambda\rightarrow \infty$ therefore proves \eqref{eqn:polynomial-bound}.
	\end{proof}
	
	We now complete the proof of Theorem \ref{thm:main}.
	\begin{proof}[Proof of Theorem~\ref{thm:main}]
		By duality, it suffices to prove \eqref{eqn:kinetic} for every finite orthonormal family. Fix $0<\alpha<\pi$. We apply Lemma~\ref{lem:polynomial} with $q(s)=b_\alpha(s)=-\alpha \tan(\alpha(s-1/2))$ as above and $\kappa=\alpha^2$. 
		
		We verify that $b_\alpha$ satisfies the stronger assumption \eqref{eqn:certificate}. For $s,t\in [0,1]$, put $X=\alpha(s-1/2)$, $Y=\alpha(t-1/2)$ and $z=X-Y=\alpha(s-t)$. Then
		\begin{align*}
			b_\alpha(s)-b_\alpha(t)=-\frac{\alpha(\sin X\cos Y-\cos X\sin Y)}{\cos X\cos Y}=-\frac{\alpha\sin z}{\cos X\cos Y}\,,\\
			b_\alpha(s)b_\alpha(t)+\alpha^2=\frac{\alpha^2(\sin X\sin Y+\cos X\cos Y)}{\cos X\cos Y}=\frac{\alpha^2\cos z}{\cos X\cos Y}\,.
		\end{align*}
		Dividing the first identity by $s-t=z/\alpha$ and adding the second yields, with $\frac{\sin z}{z}$ understood as $1$ at $z=0$,
		\begin{align}\label{eqn:balphaident}
			b_\alpha(s)b_\alpha(t)+\frac{b_\alpha(s)-b_\alpha(t)}{s-t}+\alpha^2=\frac{\alpha^2}{\cos X\cos Y}\left(\cos z-\frac{\sin z}{z}\right)\leq 0\,.
		\end{align}
		Indeed, $|X|,|Y|\leq\alpha/2<\pi/2$, so the denominator is positive. For $0<z<\pi$, we have $\sin z-z\cos z=\int_0^z t\sin t\dd t\geq0$. The expression $\cos z-\frac{\sin z}{z}$ is even and tends to zero at $z=0$, which proves the sign for every $|z|\leq\alpha$, including $s=t$.
		
		Applying Lemma~\ref{lem:polynomial} with $q=b_\alpha$ and $\kappa=\alpha^2$ gives
		\begin{align*}
			\sum_{j=1}^N\int_{\R}(u_j')^2\dd x\geq\frac{\alpha^2}{4}\int_{\R}\left(\sum_{j=1}^N u_j^2\right)^3\dd x\,.
		\end{align*}
		Letting $\alpha\uparrow\pi$ proves \eqref{eqn:kinetic} and hence Theorem~\ref{thm:main}. 
	\end{proof}
	
	\section{Proof of Theorem \ref{thm:higher-dimensional}}\label{sec:vector-proof}
	
	We now extend the one-dimensional Lieb--Thirring inequality to operator-valued potentials, retaining the sharp constant. As in Section~1, we work with the equivalent kinetic formulation, which now concerns Hilbert-space-valued orthonormal families. We take the inner product to be linear in its second argument. 
	
	\begin{proposition}\label{prop:vector-kinetic}
		Let $\psi_1,\ldots,\psi_N\in H^1(\R;\Hs)$ be orthonormal in $L^2(\R;\Hs)$, where $\Hs$ is a separable complex Hilbert space. Define the non-negative finite-rank density operator
		\begin{align*}
			R(x)f=\sum_{j=1}^N\langle\psi_j(x),f\rangle_{\Hs}\,\psi_j(x)\,,\qquad f\in\Hs\,.
		\end{align*}
		Then
		\begin{align*}
			\sum_{j=1}^N\int_{\R}\|\psi_j'(x)\|_{\Hs}^2\dd x\geq\frac{\pi^2}{4}\int_{\R}\tr_{\Hs}R(x)^3\dd x\,.
		\end{align*}
	\end{proposition}
	
	The proof of Proposition \ref{prop:vector-kinetic} is contained in Subsection \ref{sec:vector-kineticproof}. We first explain how it is used to obtain Theorem \ref{thm:higher-dimensional} using the lifting argument of \cite{laptev2000sharp,dolbeault2008improved}.

	Let $x\mapsto\mathcal V(x)$ be a strongly measurable family of non-negative self-adjoint compact operators, with
	\begin{align*}
		\int_{\R}\tr_{\Hs}\mathcal V(x)^{3/2}\dd x<\infty\,.
	\end{align*}
	Using the operator-valued duality, as in Dolbeault, Laptev, and Loss \cite{dolbeault2008improved}, Proposition~\ref{prop:vector-kinetic} gives
	\begin{align}\label{eqn:operator-spectral-main}
		\Tr_{L^2(\R;\Hs)}\left(-\frac{d^2}{dx^2}\otimes I_{\Hs}-\mathcal V\right)_-\leq\frac4{3\sqrt3\pi}\int_{\R}\tr_{\Hs}\mathcal V(x)^{3/2}\dd x\,,
	\end{align}
	where $I_{\Hs}$ is the identity on $\Hs$, $H_-=\max\{-H,0\}$, and the Schr\"odinger operator is defined by its closed, semibounded quadratic form; see, e.g., \cite{laptevReadSchimmer}.
	
	\begin{proof}[Proof of Theorem~\ref{thm:higher-dimensional}]
		For $d=1$, the assertion is Theorem~\ref{thm:main}. For $d\geq2$, take $\mathcal{H}=L^2(\R^{d-1})$ and $\mathcal{V}(x_1)=(-\Delta_{d-1}-V(x_1,\cdot))_{-}$. Using \eqref{eqn:operator-spectral-main}, we obtain 
		\begin{align*}
			\Tr(-\Delta-V)_-&\leq\Tr_{L^2(\R;\Hs)}\left(-\frac{d^2}{dx_1^2}\otimes I_{\Hs}-\mathcal V\right)_-\\&\leq\frac4{3\sqrt3\pi}\int_{\R}\Tr(-\Delta_{d-1}-V(x_1,\cdot))_-^{3/2}\dd x_1\\
			&\leq\frac{4L_{3/2,d-1}^{\mathrm{cl}}}{3\sqrt3\pi}\int_{\R^d}V(x)_+^{1+d/2}\dd x\,,
		\end{align*}
		where we have used the variational principle in the first step and the sharp $\gamma=3/2$ inequality from \cite{laptev2000sharp} in the last. We conclude that
		\begin{equation*}
			L_{1,d}\leq\frac4{3\sqrt3\pi}L_{3/2,d-1}^{\mathrm{cl}}=\frac2{\sqrt3}L_{1,1}^{\mathrm{cl}}L_{3/2,d-1}^{\mathrm{cl}}=\frac2{\sqrt3}L_{1,d}^{\mathrm{cl}}\,.\qedhere 
		\end{equation*}
	\end{proof}

	\subsection{The vector-valued extension}
	
	Let $\psi_1,\ldots,\psi_N$ be an orthonormal family as in Proposition~\ref{prop:vector-kinetic}. We define its density matrix $Q$ and cumulative mass matrix $A$ by
	\begin{align*}
		Q_{ij}(x)=\langle\psi_i(x),\psi_j(x)\rangle_{\Hs}\,,\qquad A(x)=\int_{-\infty}^xQ(y)\dd y\,.
	\end{align*}
	Both are $N\times N$ matrices. Positivity of $Q$ and orthonormality give
	\begin{align*}
		A'=Q\,,\qquad A(-\infty)=0\,,\qquad A(+\infty)=I\,,\qquad 0\leq A\leq I\,.
	\end{align*}
	Thus the cumulative mass matrix satisfies the same bounds as in the scalar proof. For the real-valued family in Section~1, $Q=UU^{\mathsf T}$; here $Q$ may have rank larger than one.
	
	To extend the argument of Section~1, it is convenient to introduce the operator $W(x):\Hs\to\C^N$ defined by
	\begin{align*}
		(W(x)f)_j=\langle\psi_j(x),f\rangle_{\Hs}\,,\qquad W^*c=\sum_jc_j\psi_j\,.
	\end{align*}
	Then $Q=WW^*$ and $R=W^*W$. Their nonzero eigenvalues agree, and hence
	\begin{align}\label{eqn:density-traces}
		\tr_{\Hs}R^3=\tr_{\C^N}Q^3\,.
	\end{align}
	We denote the kinetic energy density by
	\begin{align*}
		\tau=\sum_j\|\psi_j'\|_{\Hs}^2=\|W'\|_{\mathrm{HS}}^2\,.
	\end{align*}
	Here $\|\cdot\|_{\mathrm{HS}}$ denotes the Hilbert--Schmidt norm. In particular, $W$ belongs to $H^1$ with values in the Hilbert--Schmidt operators from $\Hs$ to $\C^N$. We write $\tr$ for a trace when its space is clear.

	As in the scalar proof, we introduce an auxiliary function $q\in C^1([0,1];\R)$ and a constant $\kappa\in\R$, the conditions on which will be specified below. We seek a real, locally absolutely continuous function $F_q$ that satisfies
	\begin{align}\label{eqn:vector-flux-target}
		\tau-F_q'\geq\frac{\kappa}{4}\tr Q^3\quad\text{a.e. on }\R\,,\qquad \lim_{x\to\pm\infty}F_q(x)=0\,.
	\end{align}
	These properties give the kinetic bound with constant $\kappa/4$ by integration.
	
	We will take
	\begin{align}\label{eqn:vector-boundary}
		F_q=\frac14\sum_{i,j}|Q_{ij}|^2q\left(\frac{\lambda_i+\lambda_j}{2}\right)\,,
	\end{align}
	where, at each $x$, the $\lambda_i$ are the eigenvalues of $A(x)$ and $Q(x)$ is expressed in an orthonormal eigenbasis of $A(x)$. This expression is independent of the chosen eigenbasis. In the real scalar case, $Q=UU^{\mathsf T}$ gives $|Q_{ij}|^2=p_ip_j$, so this recovers \eqref{eqn:boundary-term}.
	
	To state the condition on $q$, we introduce the following quantity. For an $N\times N$  matrix $P$ and $\lambda\in[0,1]^N$, put $s_{ij}=(\lambda_i+\lambda_j)/2$ and define
	\begin{align*}
		\mathcal G_{q,\kappa}(P,\lambda)=\sum_{i,j,k=1}^N\re(P_{ij}P_{jk}P_{ki})\bigl(-\kappa-q(s_{ij})q(s_{kj})-q[s_{ij},s_{kj}]\bigr)\,,
	\end{align*}
	where $q[s,t]=(q(s)-q(t))/(s-t)$, with $q[s,s]=q'(s)$.
	
	\begin{lemma}\label{lem:vector-chain}
		Suppose a function $q\in C^1([0,1];\R)$ and a constant $\kappa>0$ satisfy
		\begin{align}\label{eqn:vector-certificate}
			\mathcal G_{q,\kappa}(P,\lambda)\geq0
		\end{align}
		for every $N\times N$ non-negative self-adjoint matrix $P$, and every $\lambda\in[0,1]^N$. Then
		\begin{align}\label{eqn:vector-generic-bound}
			\sum_{j=1}^N\int_{\R}\|\psi_j'(x)\|_{\Hs}^2\dd x\geq\frac\kappa4\int_{\R}\tr_{\Hs}R(x)^3\dd x\,. 
		\end{align}
	\end{lemma}
	
	\begin{proof}
		It suffices to prove the result when $q$ is a polynomial. Indeed, approximate $q$ in $C^1([0,1])$ by real polynomials $q_n$, as in Lemma~\ref{lem:polynomial}. Uniform convergence of $q_n$ and its divided differences gives
		\begin{align}\label{eqn:vector-approximation-error}
			|\mathcal G_{q_n,\kappa}(P,\lambda)-\mathcal G_{q,\kappa}(P,\lambda)|\leq\varepsilon_n\sum_{i,j,k}|P_{ij}P_{jk}P_{ki}|
		\end{align}
		for some $\varepsilon_n\geq0$ tending to zero, independent of $P$ and $\lambda$. Since $P\geq0$, we have $|P_{ij}|\leq\sqrt{P_{ii}P_{jj}}$, and hence
		\begin{align}\label{eqn:vector-cubic-bound}
			\sum_{i,j,k}|P_{ij}P_{jk}P_{ki}|\leq(\tr P)^3\leq N^2\tr P^3\,.
		\end{align}
		Put $\kappa_n=\kappa-N^2\varepsilon_n$. Substituting $-\kappa_n=-\kappa+N^2\varepsilon_n$ into the definition of $\mathcal G_{q_n,\kappa_n}$ and using $\sum_{i,j,k}\re(P_{ij}P_{jk}P_{ki})=\tr P^3$, we obtain
		\begin{align*}
			\mathcal G_{q_n,\kappa_n}(P,\lambda)&=\mathcal G_{q_n,\kappa}(P,\lambda)+N^2\varepsilon_n\sum_{i,j,k}\re(P_{ij}P_{jk}P_{ki})\\
			&=\mathcal G_{q_n,\kappa}(P,\lambda)+N^2\varepsilon_n\tr P^3\\
			&\geq\mathcal G_{q,\kappa}(P,\lambda)-\varepsilon_n\sum_{i,j,k}|P_{ij}P_{jk}P_{ki}|+N^2\varepsilon_n\tr P^3\\
			&\geq\mathcal G_{q,\kappa}(P,\lambda)-N^2\varepsilon_n\tr P^3+N^2\varepsilon_n\tr P^3\\
			&=\mathcal G_{q,\kappa}(P,\lambda)\geq0\,.
		\end{align*}
		The first inequality follows from \eqref{eqn:vector-approximation-error}, the second from \eqref{eqn:vector-cubic-bound}, and the last by assumption. Thus $q_n$ satisfies \eqref{eqn:vector-certificate} with $\kappa_n$ for every non-negative $N\times N$ matrix $P$ and every $\lambda\in[0,1]^N$. Since $N$ is fixed, $\kappa_n\to\kappa$ and $\kappa_n>0$ for sufficiently large $n$. Applying the polynomial case gives
		\begin{align*}
			\sum_{j=1}^N\int_{\R}\|\psi_j'\|_{\Hs}^2\dd x\geq\frac{\kappa_n}{4}\int_{\R}\tr_{\Hs}R^3\dd x\,.
		\end{align*}
		Letting $n\to\infty$ gives \eqref{eqn:vector-generic-bound}, since both integrals are finite.

		Thus, we may assume that $q$ is a polynomial. Expand $q((s+t)/2)=\sum_{r,l}c_{rl}s^rt^l$, with $c_{rl}=c_{lr}$, and write
		\begin{align*}
			F_q=\frac14\sum_{r,l}c_{rl}\tr(A^rQA^lQ)\,,\qquad Q'=W'W^*+WW'^*\,.
		\end{align*}
		Since $A$ and $Q$ are locally absolutely continuous, so is $F_q$; all derivative identities below hold almost everywhere. We verify below that this polynomial expression agrees with \eqref{eqn:vector-boundary}.
		
		All sums are finite. Differentiating in the original coordinates, using $A'=Q$ and the symmetry of the coefficients, gives
		\begin{align*}
			F_q'&=\frac12\sum_{r,l}c_{rl}\re\tr(A^rQ'A^lQ)\\
			&\quad+\frac12\sum_{r,l}c_{rl}\sum_{\ell=0}^{r-1}\re\tr(A^\ell QA^{r-1-\ell}QA^lQ)\,.
		\end{align*}
		Indeed, the two differentiated factors $Q$ give equal sums by cyclicity of the trace and symmetry in $r,l$. The two differentiated powers of $A$ likewise give equal sums, and $(A^r)'=\sum_{\ell=0}^{r-1}A^\ell QA^{r-1-\ell}$ gives the second term.
		
		At a point $x$ where this identity holds, choose any orthonormal eigenbasis of $A(x)$, with eigenvalues $\lambda_i$, and express $Q$, $Q'$ and $W$ in this fixed basis. Set
		\begin{align*}
			\qquad B_{ij}=q(s_{ij})Q_{ij}\,,\qquad  s_{ij}=\frac{\lambda_i+\lambda_j}{2}\,.
		\end{align*}
		Note that $B$ is self-adjoint. Its polynomial expression in the original coordinates is $B=\sum_{r,l}c_{rl}A^rQA^l$. Evaluating the polynomial formula for $F_q$ in this basis recovers \eqref{eqn:vector-boundary}; no eigenvectors have been differentiated.
		
		The first term in the derivative is $\tfrac12\tr(BQ')=\re\tr((BW)^*W')$, since $Q'=W'W^*+WW'^*$. For the second term, evaluating the traces and factoring monomial differences gives
		\begin{align*}
			\frac12\sum_{i,j,k}\re(Q_{ij}Q_{jk}Q_{ki})\sum_{r,l}c_{rl}\sum_{\ell=0}^{r-1}\lambda_i^\ell\lambda_k^{r-1-\ell}\lambda_j^l=\frac14\sum_{i,j,k}\re(Q_{ij}Q_{jk}Q_{ki})q[s_{ij},s_{kj}]\,.
		\end{align*}
		The factor $1/2$ from the midpoint is the same as in \eqref{eqn:chain-rule}. Thus
		\begin{align}\label{eqn:vector-chain}
			F_q'=\re\tr\bigl((BW)^*W'\bigr)+\frac14\sum_{i,j,k}\re(Q_{ij}Q_{jk}Q_{ki})q[s_{ij},s_{kj}]\,.
		\end{align}
		
		Recalling $\tau=\|W'\|_{HS}^2$ and using  $\|BW\|_{\mathrm{HS}}^2=\tr(QB^2)$ and $\sum_{i,j,k}Q_{ij}Q_{jk}Q_{ki}=\tr Q^3$, completing the square in \eqref{eqn:vector-chain} gives
		\begin{align}\label{eqn:vector-square}
			\tau-F_q'-\frac\kappa4\tr Q^3&=\bigg\|W'-\frac12BW\bigg\|_{\mathrm{HS}}^2\\
			&\quad+\frac14\sum_{i,j,k}\re(Q_{ij}Q_{jk}Q_{ki})\bigl(-\kappa-q(s_{ij})q(s_{kj})-q[s_{ij},s_{kj}]\bigr)\notag\\
			&\geq\frac14\mathcal G_{q,\kappa}(Q,\lambda)\geq0\,.\notag
		\end{align}
		The last inequality follows from \eqref{eqn:vector-certificate}, applied with $P=Q(x)$ expressed in the chosen eigenbasis and with $\lambda$ the eigenvalues of $A(x)$. Indeed, $Q(x)\geq0$ and $0\leq A(x)\leq I$.

		We now integrate \eqref{eqn:vector-square} over $(-\Lambda,\Lambda)$ and let $\Lambda\rightarrow \infty$. Since $F_q$ is locally absolutely continuous,
		\begin{align*}
			\int_{-\Lambda}^\Lambda\tau\dd x-F_q(\Lambda)+F_q(-\Lambda)\geq\frac\kappa4\int_{-\Lambda}^\Lambda\tr Q^3\dd x\,.
		\end{align*}
		To show that the boundary terms tend to zero, put $\rho=\|W\|_{\mathrm{HS}}^2=\tr Q$. As in Section~1, $\rho\in L^1$ and $\rho'\in L^1$ imply $\rho(x)\to0$ at both infinities. Moreover,
		\begin{align*}
			|F_q(x)|\leq\frac14\|q\|_\infty\tr\bigl(Q(x)^2\bigr)\leq\frac14\|q\|_\infty\rho(x)^2\,,\qquad \tr Q^3\leq\rho^3\in L^1\,.
		\end{align*}
		Thus $F_q(\pm \Lambda)\to0$ as $\Lambda\to\infty$. Letting $\Lambda\to\infty$ and using \eqref{eqn:density-traces} gives \eqref{eqn:vector-generic-bound}.
	\end{proof}
	
	\subsection{Proof of Proposition~\texorpdfstring{\ref{prop:vector-kinetic}}{4}}\label{sec:vector-kineticproof}
	
	We apply Lemma~\ref{lem:vector-chain} with $q=b_\alpha$ and $\kappa=\alpha^2$, where $0<\alpha<\pi$. We verify that $b_\alpha$ satisfies \eqref{eqn:vector-certificate}. Let $P$ be any $N\times N$ non-negative self-adjoint matrix and $\lambda\in[0,1]^N$. We aim to prove
	\begin{align}\label{eqn:alpha-condition}
		\mathcal G_{b_\alpha,\alpha^2}(P,\lambda)=-\sum_{i,j,k}\re(P_{ij}P_{jk}P_{ki})\bigl[\alpha^2+b_\alpha(s_{ij})b_\alpha(s_{kj})+b_\alpha[s_{ij},s_{kj}]\bigr]\geq 0\,.
	\end{align}

	First we introduce the following
	\begin{align*}
		\theta_i=\frac\alpha2\left(\lambda_i-\frac12\right),\qquad C_{ij}=\frac{P_{ij}}{\cos(\theta_i+\theta_j)}\,,
	\end{align*}
	and
	\begin{align*}
		h(t)=\frac{\sin t}{t}-\cos t,\qquad g(x,y)=\cos(x+y)h(x-y)\,.
	\end{align*}
	Using the identity \eqref{eqn:balphaident}, namely
	\begin{align*}
		\alpha^2+b_\alpha(s_{ij})b_\alpha(s_{kj})+b_\alpha[s_{ij},s_{kj}]=-\frac{\alpha^2h(\theta_i-\theta_k)}{\cos(\theta_i+\theta_j)\cos(\theta_k+\theta_j)}\,,
	\end{align*}
	we have 
	\begin{align}\label{eqn:vector-remainder-sign}
		\mathcal G_{b_\alpha,\alpha^2}(P,\lambda)=\alpha^2\sum_{i,k}\re\bigl((C^2)_{ik}C_{ki}\bigr)g(\theta_i,\theta_k)\,.
	\end{align}
	
	The matrix $C$ remains non-negative. Indeed, set $D=\operatorname{diag}((\cos\theta_i)^{-1})$ and $T=\operatorname{diag}(\tan\theta_i)$. The cosine addition formula gives
	\begin{align*}
		C-TCT=DPD\geq0\,,
	\end{align*}
	Let $\mu$ be the smallest eigenvalue of $C$, with corresponding unit eigenvector $v$. Since $C\geq\mu I$, we have
	\begin{align*}
		\mu=\langle v,Cv\rangle\geq\langle Tv,CTv\rangle\geq\mu\|Tv\|^2\,.
	\end{align*}
	As $\|T\|< 1$, this implies that $\mu\geq0$.

	To prove that the sum in \eqref{eqn:vector-remainder-sign} is non-negative, we first consider the special case in which $g(\theta_i,\theta_k)$ is replaced by $(\sigma_i-\sigma_k)^4$, for real numbers $\sigma_1,\ldots,\sigma_N$. Set $K_{ik}=(\sigma_i-\sigma_k)^2C_{ik}$. Then
	\begin{align}\label{eqn:op-quartic}
		\sum_{i,k}\re\bigl((C^2)_{ik}C_{ki}\bigr)(\sigma_i-\sigma_k)^4=2\tr(CK^2)\geq0\,.
	\end{align}

	By linearity, it therefore suffices to express $g(\theta_i,\theta_k)$ as a positive combination of fourth powers of scalar differences. The following lemma gives this representation.
	
	\begin{lemma}\label{lem:op-fourth-powers}
		Let $x_1,\ldots,x_N\in(-\pi/4,\pi/4)$. Then there exists an integer $L\geq 1$ and real numbers $\sigma_{\ell i}$ and $a_\ell\geq 0$ such that
		\begin{align*}
			g(x_i,x_j)=\sum_{\ell=1}^L a_\ell(\sigma_{\ell i}-\sigma_{\ell j})^4,\qquad 1\leq i,j\leq N\,.
		\end{align*}
	\end{lemma}
	
	\begin{proof} 
		Set $r(x,y)=\sqrt{g(x,y)}$. The overall strategy is to write $r(x_i,x_j)$ as the squared distance of two vectors $v_i,v_j$ in an auxiliary vector space. Since $g=r^2$ this will allow us to express it in the claimed form.
		
		To construct such vectors, we aim to express $r(x_i,x_j)$ in terms of a positive bilinear form. To do so, we first derive an integral representation for $r$. Since $h(t)=t^2/3+O(t^4)$, we have
		\begin{align*}
			\lim_{y\uparrow x}\partial_xr(x,y)=\frac{\sqrt{\cos(2x)}}{\sqrt3}=\frac{\mathcal{D}(x)}2\,,\qquad \mathcal{D}(x)=2\sqrt{\cos(2x)}/\sqrt3\,.
		\end{align*}
		Define $\mathcal{K}(x,y)=-\partial_x\partial_y r(x,y)$ for $x\ne y$, and set $\mathcal{K}(x,x)=0$. For $y<x$, integration in the second variable gives
		\begin{align*}
			\partial_xr(x,y)=\frac{\mathcal{D}(x)}2+\int_y^x \mathcal{K}(x,t)\dd t\,.
		\end{align*}
		Integrating in the first variable and using $r(y,y)=0$, we obtain
		\begin{align*}
			r(x,y)=\frac12\int_y^x\mathcal{D}(s)\dd s+\int_y^x\int_y^s \mathcal{K}(s,t)\dd t\dd s\,.
		\end{align*}
		Using the symmetry of $\mathcal{K}$, we arrive at 
		\begin{align}\label{eqn:interval-representation}
			r(x,y)=\frac12\int_y^x\mathcal{D}(s)\dd s+\frac12\int_y^x\int_y^x \mathcal{K}(s,t)\dd t\dd s,\qquad y<x\,.
		\end{align}
		
		For real compactly supported step functions, define the symmetric bilinear form
		\begin{align*}
			\mathcal B(f_1,f_2)=\int_{-\pi/4}^{\pi/4} \mathcal{D}(s)f_1(s)f_2(s)\dd s+\iint_{(-\pi/4,\pi/4)^2} \mathcal{K}(s,t)f_1(s)f_2(t)\dd s\dd t\,.
		\end{align*}
		Take $x_\ast=\min_i x_i$ and put $\chi_i=\mathbf1_{(x_\ast,x_i]}$. Equation \eqref{eqn:interval-representation} gives
		\begin{align*}
			r(x_i,x_j)=\frac12\mathcal B(\chi_i-\chi_j,\chi_i-\chi_j)\,.
		\end{align*}

		We assume for now that $\mathcal B$ is positive definite and show how this gives the claimed representation. Consider the space $E=\operatorname{span}\{\chi_1,\ldots,\chi_N\}$, with functions identified up to equality almost everywhere. Our assumption makes $\mathcal B$ an inner product on $E$. If $E=\{0\}$, the assertion is immediate. Otherwise, put $m=\dim E$.
		
		Choose a basis $e_1,\ldots,e_m$ of $E$ which is orthonormal with respect to $\mathcal B$, and define the coordinate vectors $v_i$ by
		\begin{align*}
			\frac{\chi_i}{\sqrt2}=\sum_{\ell=1}^m v_{i\ell}e_\ell\,,\qquad v_i=(v_{i1},\ldots,v_{im})\in\R^m\,.
		\end{align*}
		By orthonormality,
		\begin{align*}
			|v_i-v_j|^2=\frac12\mathcal B(\chi_i-\chi_j,\chi_i-\chi_j)=r(x_i,x_j)\,,
		\end{align*}
		and hence $|v_i-v_j|^4=g(x_i,x_j)$.
		
		For $z\in \R^m$, it remains to express $|z|^4$ using scalar fourth powers. Averaging over all sign choices cancels the odd powers and gives
		\begin{align*}
			\frac1{2^m}\sum_{\epsilon\in\{-1,1\}^m}(\varepsilon\cdot z)^4=\sum_{\ell=1}^m z_\ell^4+6\sum_{\ell<k}z_\ell^2z_k^2\,.
		\end{align*}
		In $|z|^4=(\sum_\ell z_\ell^2)^2$, the mixed terms have coefficient $2$. Dividing the last identity by $3$ and adding $\tfrac23\sum_\ell z_\ell^4$ therefore gives
		\begin{align*}
			|z|^4=\frac23\sum_{\ell=1}^m z_\ell^4+\frac1{3\cdot2^m}\sum_{\epsilon\in\{-1,1\}^m}(\varepsilon\cdot z)^4\,.
		\end{align*}
		Applying this identity to $z=v_i-v_j$ expresses $g(x_i,x_j)$ as a positive combination of fourth powers of the differences $v_{i\ell}-v_{j\ell}$ and $\sum_{\ell=1}^m\varepsilon_\ell(v_{i\ell}-v_{j\ell})$.

		Finally, we establish that $\mathcal{B}$ is positive definite. Integration by parts gives, for $0<t\leq\pi/2$,
		\begin{align*}
			h(t)=\frac{t^2}{2}\int_0^1(1-u^2)\cos(tu)\,du\leq\frac{t^2}{3},\\
			0\leq\frac2t-\frac{h'(t)}{h(t)}=\frac t4\frac{\int_0^1(1-u^2)^2\cos(tu)\,du}{\int_0^1(1-u^2)\cos(tu)\,du}\leq\frac t4.
		\end{align*}
		The bounds follow from $0\leq1-u^2\leq1$ and $\cos(tu)\geq0$. Hence $0<h'/h\leq2/t$ and $0\leq4/t^2-(h'/h)^2\leq1$. Using $h''=(2/t^2-1)h$ and putting $t=|x-y|>0$, we obtain
		\begin{align*}
			\mathcal{K}(x,y)=\frac{r(x,y)}4\left[\tan^2(x+y)+\frac4{t^2}-\left(\frac{h'(t)}{h(t)}\right)^2\right]\,,\\
			0\leq \mathcal{K}(x,y)\leq\frac{t}{4\sqrt3\cos^{3/2}(x+y)}\leq\frac{\pi}{8\sqrt3}\frac{|\sin(x-y)|}{\cos^{3/2}(x+y)}\,.
		\end{align*}
		Here we used $t\leq(\pi/2)\sin t$.
		
		Fix $x\in(-\pi/4,\pi/4)$ and put
		\begin{align*}
			w(x)=\cos(2x)^{-3/4}\,,\qquad J_x(y)=\frac{\cos^{1/4}(2y)}{\sqrt{\cos(x+y)}}\,.
		\end{align*}
		By the product rule and the sine difference formula,
		\begin{align*}
			J_x'(y)&=-\frac{\sin(2y)}{2\cos^{3/4}(2y)\sqrt{\cos(x+y)}}+\frac{\cos^{1/4}(2y)\sin(x+y)}{2\cos^{3/2}(x+y)}\\
			&=\frac{\sin(x-y)}{2\cos^{3/4}(2y)\cos^{3/2}(x+y)}\,.
		\end{align*}
		The denominator is positive and $x-y\in(-\pi/2,\pi/2)$, so $J_x'$ is positive for $y<x$ and negative for $y>x$. Moreover, $J_x(y)\to0$ at both endpoints of $(-\pi/4,\pi/4)$ and $J_x(x)=\cos^{-1/4}(2x)$. Integrating separately on the two intervals therefore gives
		\begin{align*}
			\int_{-\pi/4}^{\pi/4}|J_x'(y)|\dd y=\int_{-\pi/4}^xJ_x'(y)\dd y-\int_x^{\pi/4}J_x'(y)\dd y=2\cos^{-1/4}(2x)\,.
		\end{align*}
		The integrals are understood as limits over interior intervals. The preceding bound on $\mathcal{K}$ gives $\mathcal{K}(x,y)w(y)\leq\pi|J_x'(y)|/(4\sqrt3)$, and hence
		\begin{align}\label{eqn:weighted-kernel-bound}
			\int_{-\pi/4}^{\pi/4} \mathcal{K}(x,y)w(y)\dd y\leq\frac{\pi}{4\sqrt3}\int_{-\pi/4}^{\pi/4}|J_x'(y)|\dd y=\frac{\pi}{2\sqrt3}\cos^{-1/4}(2x)=\frac\pi4 \mathcal{D}(x)w(x)\,.
		\end{align}
		
		We now bound the quadratic term involving $\mathcal{K}$. The inequality $2ab\leq a^2+b^2$, applied with $a=|f(x)|\sqrt{w(y)/w(x)}$ and $b=|f(y)|\sqrt{w(x)/w(y)}$, gives
		\begin{align*}
			2|f(x)f(y)|\leq f(x)^2\frac{w(y)}{w(x)}+f(y)^2\frac{w(x)}{w(y)}\,.
		\end{align*}
		Multiply by $\mathcal{K}(x,y)\geq0$ and integrate. The two terms on the right have equal integrals by symmetry of $\mathcal{K}$, so \eqref{eqn:weighted-kernel-bound} yields
		\begin{align*}
			\bigg|\iint_{(-\pi/4,\pi/4)^2}\mathcal{K}(x,y)f(x)f(y)\dd x\dd y\bigg|&\leq\int_{-\pi/4}^{\pi/4}\frac{f(x)^2}{w(x)}\left(\int_{-\pi/4}^{\pi/4} \mathcal{K}(x,y)w(y)\dd y\right)\dd x\\
			&\leq\frac\pi4\int_{-\pi/4}^{\pi/4} \mathcal{D}(x)f(x)^2\dd x\,.
		\end{align*}
		Consequently $\mathcal B(f,f)\geq(1-\pi/4)\int_{-\pi/4}^{\pi/4} \mathcal{D}f^2$. Since $\mathcal{D}>0$ on $(-\pi/4,\pi/4)$, $\mathcal B$ is positive definite modulo equality almost everywhere. 
	\end{proof}
	
	\begin{proof}[Proof of Proposition~\ref{prop:vector-kinetic}]
		Fix $0<\alpha<\pi$. Applying Lemma~\ref{lem:op-fourth-powers} with $x_i=\theta_i$, positivity follows term by term from \eqref{eqn:op-quartic}:
		\begin{align}\label{eqn:op-G-positive}
			\mathcal G_{b_\alpha,\alpha^2}(P,\lambda)=\alpha^2\sum_{i,k}\re\bigl((C^2)_{ik}C_{ki}\bigr)g(\theta_i,\theta_k)\geq 0\,.
		\end{align}
		
		Equations~\eqref{eqn:vector-remainder-sign} and \eqref{eqn:op-G-positive} give $\mathcal G_{b_\alpha,\alpha^2}(P,\lambda)\geq0$ for every finite non-negative self-adjoint matrix $P$ and every $\lambda\in[0,1]^N$, where $N$ is the size of $P$. This verifies \eqref{eqn:vector-certificate} for $q=b_\alpha$ and $\kappa=\alpha^2$. Lemma~\ref{lem:vector-chain} gives the kinetic bound with coefficient $\alpha^2/4$. Letting $\alpha\uparrow\pi$ gives the stated inequality.
	\end{proof}
	
	\subsection*{Acknowledgments}
	The authors would like to thank Rupert L. Frank, Jan Philip Solovej, Dirk Hundertmark and Charlotte Dietze for helpful discussions and advice. 
	
	\subsection*{Use of AI}
	
	A first version of the main argument was developed by GPT-6 Astra under the guidance of L.R., following collaboration over a couple of months with successive models, beginning with GPT-5.5. Extensions and improvements were then jointly obtained by M.R.S. and L.R. using Fable 5.1 and GPT-6, respectively. The argument was further simplified with assistance from the models, then restructured and verified by the authors. The authors take full responsibility for the final version of the proof.

	\subsection*{Funding}
	Support was provided by the European Research Council (ERC) under the European Union’s Horizon 2020 research and innovation programme (grant agreement No. 101097172 – GEOEDP, L.R.), the ERC Advanced Grant MathBEC (grant agreement No. 101095820, M.R.S.), and the VILLUM Foundation (grant No. VIL73411, M.R.S.).

\end{document}